\documentclass[12pt,reqno]{amsart}
\usepackage{amssymb}
\usepackage{graphicx,psfrag}
\usepackage[mathscr]{eucal}
\usepackage{xspace}
\newcommand{\R}{\mathbb{R}}

\newcommand{\xx}{\mathbf{x}}
\newcommand{\kk}{\mathbf{k}}
\newcommand{\ff}{\mathbf{f}}
\newcommand{\mm}{\mathbf{m}}

\newcommand{\uu}{\mathbf{u}}

\newcommand{\ga}{\alpha}
\newcommand{\gb}{\mathbf{\beta}}
\newcommand{\gc}{\gamma}
\newcommand{\gD}{\Delta}
\newcommand{\po}{\partial}

\newcommand{\ve}{\varepsilon}

\newcommand{\gd}{\delta}

\newcommand{\gl}{\lambda}

\newcommand{\Om}{\Omega}

\newcommand{\grad}{\nabla}

\newcommand{\bea}{\begin{eqnarray}}
\newcommand{\eea}{\end{eqnarray}}

\theoremstyle{plain}
\newtheorem{theorem}{Theorem}[section]

\newtheorem{lemma}{Lemma}[section]
\newtheorem{proposition}{Proposition}[section]
\theoremstyle{definition}

\theoremstyle{remark}

\newtheorem{remark}{Remark}[section]
\numberwithin{equation}{section}

\begin{document}

\title[Subsonic Flows for Full Euler Equations]
{Subsonic Flows for the Full Euler
Equations in Half Plane}
\author
{Jun Chen  }
\address{J. Chen, Department of Mathematics\\
         University of Wisconsin-Madison\\
         Madison, WI 53706-1388, USA}
\email{jchen@math.wisc.edu}

\date{}



\begin{abstract}
We study the subsonic flows governed by full Euler equations in the
half plane bounded below by a piecewise smooth curve asymptotically
 approaching $x_1$-axis. Nonconstant conditions in the
  far field are prescribed
to ensure the real Euler flows. The Euler system is reduced to a single
elliptic equation for the stream function. The existence, uniqueness
 and asymptotic
behaviors of the solutions for the reduced equation are established
by Schauder fixed point argument and some delicate estimates. The
existence of subsonic flows for the original Euler system is proved
 based on the results for the reduced equation, and their
 asymptotic behaviors in the far
field are also obtained.

\end{abstract}

\maketitle

\section{Introduction}

In this paper, we study subsonic polytropic flows governed by
two-dimensional steady, full Euler equations:
\begin{equation}
\left\{\begin{aligned}
& \nabla\cdot  \mm=0, \\
&\nabla\cdot\left(\frac{\mm\otimes\mm}{\rho}\right)
  +\nabla p=0,\\
&\nabla\cdot\left(\mm(E+p/\rho)\right)=0,
\end{aligned}\right.
\label{eqn-Euler1}
\end{equation}
where $\grad$ is the gradient in $\xx=(x_1,x_2)\in\R^2$, $\mm=(m_1,
m_2)$ the momentum, $\rho$ the density, $p$ the pressure, and
$$
E=\frac{|\mm|^2}{2 \rho^2}+\frac{p}{(\gamma-1)\rho}
$$
the energy with adiabatic exponent $\gamma>1$. The sonic speed of
the flow is defined by
$$
c=\sqrt{\gamma p/\rho}.
$$
The flow is said to be subsonic if $|\mm / \rho| < c$.

To my best knowledge, no theoretical result was obtained for
subsonic flows governed by full Euler system in unbounded domain.
There are rich literatures of subsonic potential flows, which is a
simplified model for Euler flows. Shiffman obtained the first
existence result in \cite{shiffmen}, using variational method. In
\cite{bers}, Bers used complex analysis to show existence and
uniqueness for the subsonic potential flows. Finn and Gilbarg
\cite{fg,fg2} solved the problem by PDE approach. Recently,
Chen-Dafermos-Slemord-Wang \cite{cdsw} pushed the subsonic flows to the
sonic limit, using the framework of compensated
compactness. With in the same framework, Chen-Slemord-Wang
 \cite{csw} obtained transonic solutions by a vanishing viscosity
method. Other results for subsonic or transonic flows of
various models can be found in
\cite{ckk1,ckk2,ccf,ccs,cm,cm2,cm3,cm4,cm5,schen1,schen2,hk,gm,M,serre,xy,yzheng,yzheng2},
and \cite{cofr,dong,fried} provide related background and
introduction.

The domain we study is the upper half plane with piecewise smooth
boundary asymptotically flattened as $|x_1| \to \infty$. This
setting can be viewed as a symmetric airfoil problem. For a airfoil
symmetric about $x_1$-axis, we cut the exterior domain in half along
the symmetry axis, and the upper part becomes our domain. More
generally, we allow the boundary to be curved away from the profile.
In this way, our setting also includes the model for the wind
glancing the landscape.

Let $U=(\mm, p, \rho)$ be the solution for the subsonic flow. We
prescribe an asymptotic limit $U_\infty$, close to a constant
subsonic state $U_0$,  for the flow in  the far field. Unlike the
setting for potential flows, the asymptotic behavior $U_\infty$ is
not a constant state. Otherwise, the full Euler system can be
reduced to a potential flow (cf. Proposition \ref{prop-potential}).
To guarantee the convergence of the flows to $U_\infty$ in the far
field, we need to obtain some decay property for $\psi -l$, which is
the difference between the stream function and its limit behavior.
For the whole plane, one knows that the fundamental solution for a
Laplace equation has the form $\log |\xx|$. Therefore, in general,
we can not expect the decay of a solution for an elliptic equation
as $|\xx| \to \infty$ in the exterior domain of the whole plane.
This is the main technical obstacle for us to obtaining the subsonic
flows in the whole plane. However, when flows are restricted in the
half plane, we can exclude the logarithmic growth of  solutions by
prescribing proper decay condition at the infinity.

 We reduce the
Euler system to a single elliptic equation \eqref{psi-elliptic} for
a stream function $\psi$ by capturing some conservation properties
of the system. More precisely, three properties are contained in
\eqref{psi-elliptic}: (1) existence of $\psi$ represents
conservation of mass; (2) we use the fact $p/\rho^\gc$ is constant
along streamlines during the reduction, which implies entropy is
conserved on each streamline. (3) to solve for $\rho$ in terms of
$\psi, \grad \psi$, we use Bernoulli's law, which relates to the
conservation of energy. Usually, stagnation points occur in various
situations and cause major difficulties (for instance, regular
reflection for Euler equations in self-similar coordinates). Our
reduction process enables us to bypass the difficulty and to obtain
the existence of solutions. However, we do not have uniqueness
result due to the existence of stagnation points and complex
behaviors of streamlines. More details are explained in Remark
\ref{rem-unique}.

Once the Euler system is reduced to the elliptic equation
\eqref{psi-elliptic}, the remaining work is to solve this nonlinear
equation. In detail, we first truncate the original domain $\Om$ by
ball $B_R (O)$ centered at the origin with radius $R$. We solve
\eqref{psi-elliptic} in bounded domain $\Om_R = \Om \cap B_R (O)$
with properly prescribed boundary condition. It is a standard method
that we linearize \eqref{psi-elliptic}, construct a map $T$ by
solving the linearized equation, and prove the existence of a fixed
point for $T$ by Schauder fixed point theorem. The fixed point
$\psi_R$ is the solution for \eqref{psi-elliptic} in domain $\Om_R$.

To take the limit of $\{\psi_R\}$ as   $R \to \infty$ and obtain the
solution in $\Om$, the estimates should be independent of the radius
$R$. It makes the estimates complicated  that there is no sign for
the coefficient $b_0$ in the linear equation \eqref{eqn-lin-psi}. By
choosing a proper barrier and using maximum principle with no
restriction on the sign of $b_0$ (Lemma \ref{lem-compare}), we
obtain uniform estimates, independent of $R$, for $\psi_R$. The
barrier function we construct only works for the half plane, not the
whole plane.
 Whether one can find a suitable barrier function for the whole plane
is unclear at this moment.

The rest of the paper is organized as follows. In Section 2,
we set up the subsonic flow problem, introduce the weighted norms, and
state the main result. Section 3 explains the reduction of the full Euler system to
a single elliptic equation for the stream function $\psi$. In Section 4, we
prove a technical lemma for the linearized equation and obtain crucial esitmates.
In Section 5, we construct a
iteration scheme to solve the nonlinear equation  \eqref{psi-elliptic}
 in truncated domain
$\Om_R$. Schauder fixed point argument is used to prove the existence of the solution.
In Section 6, we take the limit of subsequence of solutions $\psi_R$ in $\Om_R$ and obtain
the solution in the half plane $\Om$. The relation between the original Euler system
 and the reduced system is explained.

\section{Setup of the Subsonic Flow Problem}


%
%
%
%

 In order to describe the conditions and results of our subsonic
 problem,  we need to use the following weighed H\"older norms:
for any $\xx, \xx' $ in a two-dimensional domain $\Om$ and for a
subset $P$ of $\po \Om$, define $\gd_\xx := \min
(\mbox{dist}(\xx,P),1)$, $\gd_{\xx,\xx'} := \min (\gd_\xx,
\gd_{\xx'},1)$, $\gD_\xx := \max(|\xx|,1)$ and $\gD_{\xx, \xx'}:=
\max(|\xx|, |\xx'|,1)$ . Let $\ga \in (0,1)$,  $\sigma, \beta \in
\mathbb{R}$, and $k$ be a nonnegative integer. Let $\kk = (k_1,
k_2)$ be a integer-valued vector, where $k_1, k_2 \ge 0$, $|\kk|=k_1
+k_2$ and $D^{\kk}= \po_{x_1}^{k_1}\po_{x_2}^{k_2}$. We define
\begin{eqnarray}
\nonumber
 &&[ u ]_{k,0;(\beta);\Om}^{(\sigma;P)}
  = \sup_{\tiny \begin{array}{c}
\xx\in \Om\\
 |\kk|=k
\end{array}}(\gd_{\xx}^{\max(k+\sigma,0)}\gD_\xx^{\gb+k}|D^\kk u(\xx)|), \\
 && {[ u ]}_{k,\ga;(\beta);\Om}^{(\sigma;P)}=  \sup_{\tiny
\begin{array}{c}
   \xx, \xx'\in \Om\\
 \xx \ne \xx'\\
|\kk|=k
 \end{array}}
\Big(\gd_{\xx,\xx'}^{\max(k+\ga+
\sigma,0)}\gD_{\xx,\xx'}^{\beta+k+\ga}\frac{|D^\kk u(\xx)-D^\kk
u(\xx')|}{|\xx-\xx'|^\ga}\Big),
 \nonumber\\
&& \|u\|_{k,\ga;(\gb);\Om}^{(\sigma;P)}= \sum_{i=0}^k {[ u ]}_{i,0
;(\gb);\Om}^{(\sigma;P)} + {[ u ]}_{k,\ga;(\gb);\Om}^{(\sigma;P)}.
\label{def-norm}
\end{eqnarray}

For a vector-valued function $\uu=(u_1, u_2, \cdots,u_n )$, we
define
\[
 \|\uu\|_{k,\ga;(\gb);\Om}^{(\sigma;P)} =
\sum_{i=1}^n  \|u_i\|_{k,\ga;(\gb);\Om}^{(\sigma;P)}.
\]

\begin{remark}
In the definition of the weighted norms, the lower index in the
parenthesis represents the weigh at the infinity and the upper index
represents the weight to the set $P$, which will be the set of some
corner points on the boundary in the paper.
\end{remark}

Define
 \begin{equation}\label{def-spaceHolder}
C^{(\sigma;P)}_{k,\ga;(\gb)}(\Om)= \{ u:
\|u\|_{k,\ga;(\gb);\Om}^{(\sigma;P)} < \infty\}.
\end{equation}

For the weighted norms of functions in one-dimensional space
$\Gamma=(a,b)$, with either $a=-\infty$ or $b=\infty$, we define
\begin{eqnarray}
\nonumber
  [f]_{k,0;(\gb);\Gamma} &=& \sup_{x\in \Gamma}(|x|+1)^{k+\gb}|f^{(k)}(x)| \\
  \nonumber
   [f]_{k,\ga;(\gb);\Gamma}&=&  \sup_{x , x'\in \Gamma, x\ne x'}
   (\max(|x|,|x'|)+1)^{k+\ga+\gb}\frac{|f^{(k)}(x)-f^{(k)}(x') |}{|x-x'|^{\ga}} \\
  \|f\|_{k,\ga;(\gb);\Gamma} &=&
  \sum_{i=0}^{k}[f]_{k,0;(\gb);\Gamma} + [f]_{k,\ga;(\gb);\Gamma}.
  \label{def-norm-f}
\end{eqnarray}

Our domain $\Om$ is the upper half  plane bounded below by a
piecewise smooth curve consisting of three parts:
\begin{equation}\label{def-bd}
    \po \Om = \Gamma_- \cup \mathcal{A} \cup \Gamma_+.
\end{equation}
The following is the description of the three parts for $\po \Om$ (see figure
\ref{fig-domain}).
\begin{figure}[htp]
\centering
\psfrag{Om}{\Large $\Om$}
\psfrag{Ga-}{ $\Gamma_-$}
\psfrag{Ga+}{ $\Gamma_+$}
\psfrag{A+}[c]{ $A_+$}
\psfrag{A-}[c]{ $A_-$}
\psfrag{A}[b]{ $\mathcal{A}$}
\psfrag{Streamlines}[b]{ Streamlines}
\includegraphics[width=13cm]{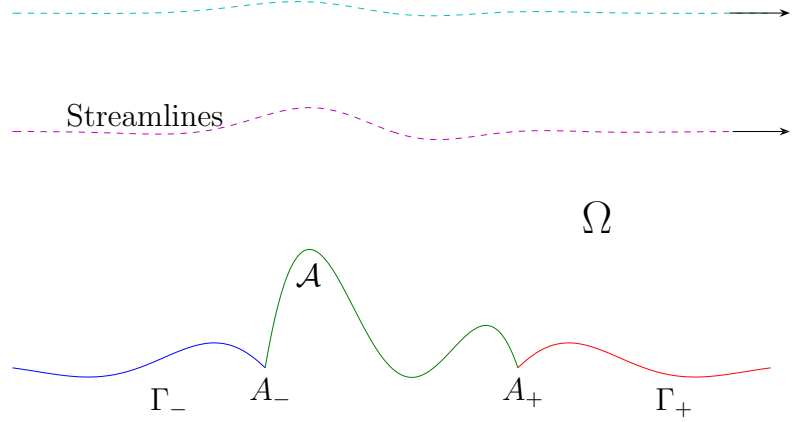}
\caption{Domain $\Om$ for Subsonic Flows}\label{fig-domain}
\end{figure}
Let $\Gamma_{\pm} =\{ x_2 = f_{\pm}(x_1)\}$, where $f_-$ is defined
on $(-\infty, -1)$ and $f_+$ is defined on $(1,\infty)$. Both
$\Gamma_-$ and $\Gamma_+$ approach $x_1$-axis as $|x_1|$ tends to
$\infty$. More precisely, we let
\begin{eqnarray}
\label{con-f-}
&& \|f_-\|_{2,\ga;(\ga+\gb);(-\infty,-1)}\le 1, \\
&&\|f_+\|_{2,\ga;(\ga+\gb);(1,\infty)}\le 1.\label{con-f+}
\end{eqnarray}

Let $A_- = (-1, f_-(-1))$ and $A_+ = (1, f_+ (1))$ be the end points
of $\Gamma_-, \Gamma_+$, respectively. The arch $\mathcal{A}$
connecting $\Gamma_{\pm}$ at $A_{\pm}$ can be parameterized by
\begin{equation}\label{def-profile}
    \ff(s)=(f_1(s), f_2(s)), \quad s\in (-1,1),
\end{equation}
where $\ff(-1) =A_-, \ff(1)=A_+$ and $f_1, f_2$ are $C^{2,\ga}$
smooth functions.

We assume the angles $\theta^0_{\pm}$ between $\Gamma_{\pm}$ and
$\mathcal{A}$ at points $A_{\pm}$ satisfies:
\begin{equation}\label{con-angles1}
   \gd < \theta^0_{\pm} < \pi -\gd,
\end{equation}
for a fixed constant $\gd$.

\begin{remark}
The above condition guarantees that stream function $\psi$ for the
flow is $C^{1,\ga}$ up to the corner points, which means the flow
$U$ is $C^{\ga}$ up to the corners. If we allow corner angles
$\theta_{\pm}^0 \ge \pi$, $\psi$ will be $C^{\ga}$ up to the
corners, and $U$ will blow up at corners. We exclude the latter
situation just to avoid unimportant details.
\end{remark}

Without loss of generality, we may also assume that
\begin{equation}\label{con-bd-profile}
    \mathcal{A} \subset B_{D_0}(0),
\end{equation}
i.e., $\mathcal{A}$ is  contained in  the ball of radius $D_0$
centered at $0$, and
\begin{equation}\label{con-boundary-above}
    f_-(x_1)>-\frac{1}{2},\ \ f_+(x_1)>-\frac{1}{2},\ \
f_2(s)>-\frac{1}{2}.
\end{equation}
This means that domain $\Om$  is above the line $x_2 =-\frac{1}{2}$.

We prescribe slip condition on the boundary:
\begin{equation}\label{con-slip-U}
    \mm \cdot \nu|_{\po \Om} =0,
\end{equation}
where $\nu$ is the outer normal on boundary $\po \Om$.

Let $(m_\ast, 0, p_0, \rho_0)$ be a constant subsonic solution for
\eqref{eqn-Euler1}, i.e.,  $m_\ast/\rho_0$ less than $
\sqrt{p_0/\rho_0}$. Fix constants $p_0, \rho_0$ and let $m_0 \le
m_\ast$ be a sufficiently small constant to be determined later. So
$U_0=(m_0, 0, p_0, \rho_0)$, as our background state, is also a
subsonic solution for \eqref{eqn-Euler1}.  We define a vector-valued
function $U_\infty=(m_\infty, 0, p_0, \rho_\infty)$ of variable
$x_2$  as the asymptotic state for our solution $U=(\mm,p, \rho)$ at
the far field. We assume that $U_\infty$ is a small perturbation of
the background solution $U_0$:
\begin{equation}\label{con-Uinfty}
    \|U_\infty - U_0\|_{2,\ga;(0); (0,\infty)} \le \ve m_0,
\end{equation}
where $ 0<\ve <\frac{1}{2} $ is a  small parameter to be determined
later.

Set $P= \{A_-, A_+\}$ as the set for the weight.

Now we state our main theorem about the existence of the subsonic
flow in the half plane $\Om$:

\medskip

\begin{theorem} \label{thm-main}
Suppose the boundary $\po \Om$ satisfies
\eqref{con-f-}--\eqref{con-boundary-above} and $U_\infty$ satisfies
\eqref{con-Uinfty}. We fix $0< \gb< \ga<1$, depending on $\gd$ in
\eqref{con-angles1}. For sufficiently small $m_0$, depending on
$m_\ast, p_0, \rho_0, \gd, \ga, \gb$ and the profile $\mathcal{A}$,
and sufficiently small $\ve$, depending on $m_\ast, p_0, \rho_0,
\gd, \ga, \gb, \mathcal{A}$ and $m_0$, there exists a subsonic
solution $U \in C^{(-\ga;P)}_{1,\ga;(\gb)}(\Om) $ for
\eqref{eqn-Euler1} with boundary condition \eqref{con-slip-U}, such
that
\begin{equation}\label{est-U-Uinfty}
    \|U- U_\infty \|^{(-\ga;P)}_{1,\ga;(\gb);\Om} \le C m_0,
\end{equation}
where $C$ is a constant only depending on $m_\ast, p_0, \rho_0, \gd,
\ga, \gb$ and the profile $\mathcal{A}$, but independent of $m_0$
and $\ve$.

\end{theorem}

\begin{remark}
Estimate \eqref{est-U-Uinfty} immediately gives the asymptotic
behavior of the flow $U$. That is $U$ approaches $U_\infty$ in $C^0$
norm at the rate $|\xx|^{-\gb}$ as $|\xx| \to \infty$.
\end{remark}

\section{Reduction of the Euler System} \label{sec-reduce-system}
In this section, we use the conservation properties of the Euler
equations \eqref{eqn-Euler1} to reduce the four-equation system to
one elliptic equation.

By the conservation of mass (first equation of \eqref{eqn-Euler1}),
we can find a potential function $\psi$ for the vector field $(-m_2
, m_1)$, i. e.,
 \begin{equation}\label{def-stream}
 \psi_{x_1} = -m_2, \ \ \psi_{x_2} = m_1.
\end{equation}

From \eqref{eqn-Euler1}, we can derive
\begin{equation}\label{eqn:gamma-law}
(m_1 \po_{x_1} + m_2 \po_{x_2})(\gamma\ln \rho - \ln p) =0,
\end{equation}
which implies that the quantity $\rho^\gc/p$ is constant along
streamlines, provided that the solution is $C^1$ smooth. This
constant only depends on the stream function $\psi$ . Thus, we have
\begin{equation}\label{gamma-law}
    p=\frac{\gc-1}{\gc} A(\psi) \rho^\gc.
\end{equation}
We will determine function $A$ later by $U_\infty$.

 From \eqref{eqn-Euler1}, we can also
derive the Bernoulli's law:

\begin{equation}\label{bernoulli}
    \frac{|\mm|}{2\rho^2}  + \frac{\gamma p}{(\gamma-1)\rho} =B
\end{equation}
 along the streamlines, where $B$ is the Bernoulli constant depending on $\psi$.
  With equation (\ref{gamma-law}) and (\ref{def-stream}),
the Bernoulli's law (\ref{bernoulli}) can be written as
\begin{equation}\label{bernoulli-law}
    \frac{1}{2} |\grad \psi|^2 + A(\psi)\rho^{\gc+1}
    =B(\psi)\rho^2,
\end{equation}

In the subsonic region, we have
\begin{equation}\label{subsonic-conditon}
    |\grad \psi|^2 < c^2 \rho^2 = (\gc-1) A(\psi) \rho^{\gc +1}.
\end{equation}

Inequality (\ref{subsonic-conditon}) and the Bernoulli's law
(\ref{bernoulli-law}) implies

\begin{equation}\label{sub-cond-rho}
    \rho^{\gc-1} > \frac{2 B}{(\gc+1) A}.
\end{equation}

 Let $\chi = \frac{1}{2}  |\grad \psi|^2$ and $h(\rho, \psi ) = B(\psi)\rho^2
-A(\psi) \rho^{\gc+1}$. Therefore, in subsonic region,
\[
\frac{\po h }{\po \rho} =  \rho A\left(\frac{2 B}{(\gc+1)
A}-\rho^{\gc-1}\right) < 0.
\]
Hence, we can uniquely solve
\begin{equation} \label{eqn:rho}
\chi =h(\rho, \psi)
    \equiv B(\psi)\rho^2-A(\psi)\rho^{\gc+1}
\end{equation}
for $\rho= \rho(\chi, \psi)$ by implicit function theorem.


From (\ref{eqn:rho}), we can easily calculate
\begin{eqnarray}
\label{def-rhoDXi}
  \rho_\chi &=& -\frac{1}{(\gc+1) A \rho^\gc - 2 B \rho }, \\
\label{def-rhoDpsi}
  \rho_\psi &=& \frac{B' \rho - A' \rho^\gc}{(\gc+1)A \rho^{\gc-1} - 2B
  }.
\end{eqnarray}

Therefore, we compute
\begin{eqnarray}
\nonumber
  \rho_{x_1} &=& \rho_\chi (\psi_{x_1} \psi_{x_1x_1}+ \psi_{x_2}\psi_{x_1x_2} )
  + \rho_\psi \psi_{x_1}\\
   &=& \frac{ -\psi_{x_1} \psi_{x_1x_1}-\psi_{x_2}\psi_{x_1x_2} +\psi_{x_1}(B'\rho^2
   - A'\rho^{\gc+1}) }{ (\gc+1) A \rho^\gc - 2 B \rho } , \label{rho-x} \\
  \rho_{x_2} &=& \rho_\chi (\psi_{x_1} \psi_{x_1x_2}+ \psi_{x_2}\psi_{x_2x_2} )
  + \rho_\psi \psi_{x_2} \nonumber\\
   &=&\frac{ -\psi_{x_1} \psi_{x_1x_2}-\psi_{x_2}\psi_{x_2x_2} +\psi_{x_2}(B'\rho^2
   - A'\rho^{\gc+1}) }{ (\gc+1) A \rho^\gc - 2 B \rho }.
   \label{rho-y}
\end{eqnarray}

 Now we can reduce the Euler system into one
equation. We replace $\mm$ in the second equation of
\eqref{eqn-Euler1} with $(-\psi_{x_2}, \psi_{x_1})$ according to
(\ref{def-stream}). Multiplying the second equation of
\eqref{eqn-Euler1}
 by
$(\gc+1) A \rho^\gc - 2 B \rho$, and using the expressions
(\ref{rho-x}) and (\ref{rho-y}), we obtain the following equation:
\begin{equation}\label{eqn-M1}
    \psi_{x_1}(a_{ij}(\psi, \grad \psi) \psi_{x_ix_j}-F(\psi, \grad
    \psi))=0,
\end{equation}
where \begin{eqnarray}
   \label{def-a11}    && a_{11}(\psi, \grad \psi)=
    (\gc-1)A(\psi)\rho^{\gc+1}-\psi_{x_2}^2 \\
     \label{def-a12}   && a_{12}(\psi, \grad \psi)
     = a_{21}(\psi, \grad \psi)= \psi_{x_1}\psi_{x_2}\\
\label{def-a22}         && a_{22}(\psi, \grad \psi)=
    (\gc-1)A(\psi)\rho^{\gc+1}-\psi_{x_1}^2 \\
 \label{def-F}       &&F(\psi, \grad \psi)=\frac{\gc-1}{\gc}\rho^{\gc+3}
 (\gc A B' -2A'B +
AA' \rho^{\gc-1}).
      \end{eqnarray}

%
%
%

Similarly, the third equation of \eqref{eqn-Euler1} gives rise to
\begin{equation}\label{eqn-M2}
 \psi_{x_2}(a_{ij}(\psi, \grad \psi) \psi_{x_ix_j}-F(\psi, \grad
    \psi))=0.
\end{equation}

 For a system without stationary points, i.e., $\grad \psi $ is
nowhere $\mathbf{0}$, the original Euler system \eqref{eqn-Euler1}
can be reduced to the following equation for subsonic flows:

\begin{equation}\label{psi-elliptic}
a_{ij}(\psi, \grad \psi) \psi_{x_ix_j}=F(\psi, \grad
    \psi).
\end{equation}

Equation \eqref{psi-elliptic} can be written in divergence form:

\begin{equation}\label{psi-div}
    \grad \cdot \left(\frac{\grad \psi}{\rho}\right) = B' \rho - \frac{1}{\gc} A'
    \rho^\gc.
\end{equation}

Now we use the limit function $U_\infty$ to determine $A,B$ and the
limit function $l(x_2)$ of the stream function $\psi$ as $|x_1| \to
\infty $.

Define
\begin{equation}\label{def-l}
    l(x_2) = \int_0^{x_2} m_\infty (s)ds.
\end{equation}
By \eqref{con-Uinfty}, we know that
\[
\frac{1}{2} m_0 < m_\infty = l' < 2 m_0,
\]
which implies that $l$ is invertible and
\begin{equation}\label{est-l-asy}
    \frac{1}{2} m_0 x_2 < l(x_2) < 2 m_0 x_2.
\end{equation}

Let
\begin{eqnarray*}
 &&\bar{A} (x_2)= \frac{\gc p_0}{(\gc -1) \rho_\infty (x_2)}\\
 && \bar{B} (x_2) = \frac{m_\infty ^2 (x_2)}{2 \rho_\infty^2 (x_2)}
 + \frac{\gc p_0}{\rho_\infty  (x_2)}.
 \end{eqnarray*}

Then define
\begin{equation}\label{def-AB}
    A(s) = \bar{A}(l^{-1} (s)), \quad B(s)= \bar{B}(l^{-1}(s)).
\end{equation}

To describe the properties of $A$ and $B$, we need to modify the
weighted norm in \eqref{def-norm-f} as follows:
\begin{eqnarray}
\nonumber
  [f]'_{k,0;(\gb);\Gamma} &=& \sup_{x\in \Gamma}(|x|+m_0)^{k+\gb}|f^{(k)}(x)| \\
  \nonumber
   [f]'_{k,\ga;(\gb);\Gamma}&=&  \sup_{x , x'\in \Gamma, x\ne x'}
   (\max(|x|,|x'|)+m_0)^{k+\ga+\gb}\frac{|f^{(k)}(x)-f^{(k)}(x') |}{|x-x'|^{\ga}} \\
  \|f\|'_{k,\ga;(\gb);\Gamma} &=&
  \sum_{i=0}^{k}[f]'_{k,0;(\gb);\Gamma} + [f]'_{k,\ga;(\gb);\Gamma}.
  \label{def-norm'f}
\end{eqnarray}

Basically, we replace $1$ in the weight in \eqref{def-norm-f} with
$m_0$ for the scaling reason.

Set \[ A_0 = \frac{\gc p_0}{(\gc-1)\rho_0^{\gc}},
 \ \ B_0 = \frac{m_0^2}{2 \rho_0^2} + \frac{\gc p_0}{(\gc-1)\rho_0}.\]
By \eqref{con-Uinfty}, we conclude that
\begin{eqnarray}
\label{est-A}
&& \|A -A_0 \|'_{2,\ga;(0);(0,\infty)} \le C_0\ve m_0,\\
&& \|B -B_0 \|'_{2,\ga;(0);(0,\infty)} \le C_0\ve m_0,\label{est-B}
 \end{eqnarray}
where $C_0$ is a constant depending only on $m_\ast, p_0, \rho_0$.

Let us discuss the asymptotic behavior of $U$  as $|x_1| \to
\infty$. We do not expect constants states at the infinity for
general subsonic flows governed by full Euler equations. Actually,
if the flow is uniform at the infinity, we only get a potential
flow. This fact is described by the following proposition.

\begin{proposition} \label{prop-potential}
Suppose $U$ is a $C^1$ solution of \eqref{eqn-Euler1} with no
stagnation points ($\mm$ is nowhere $\mathbf{0}$). If  $U_\infty$ is
a constant state,  the flow $U$ is potential, i.e., the velocity
$\uu = \mm/\rho$ is irrotational.
\end{proposition}
\begin{proof}
If  $U_\infty$ is constant, we immediately get $A, B$ are constants
by the procedure of obtaining $A,B$. Hence, we have $A'=B'=0$ .
Equation \eqref{psi-div}, which is equivalent to \eqref{eqn-Euler1}
under the assumption in the proposition, becomes
\[
\grad \cdot \left(\frac{\grad \psi}{\rho}\right) = 0.
\]
Since \[ \psi_{x_1} = -m_2=  -\rho u_2,\ \  \psi_{x_2} = m_1 = \rho
u_1,\] the above equation is just the irrotationality condition for
the velocity $ {(u_1)}_{x_2} - {(u_2)}_{x_1}=0 $. Therefore, we have
a potential flow.
\end{proof}

In general, the Euler system \eqref{eqn-Euler1} and equation
\eqref{psi-elliptic} are not equivalent, because the streamlines may
not be nice enough for us to do the reduction of the system.
However, the solution of \eqref{psi-elliptic} guarantees the
existence of the solution for Euler equations \eqref{eqn-Euler1}.
Therefore, we only need to solve \eqref{psi-elliptic} in order to
prove Theorem \ref{thm-main}.

By the definition of the stream function $\psi$, the slip condition
\eqref{con-slip-U} becomes the Dirichlet boundary condition for
equation \eqref{psi-elliptic}:
\begin{equation}\label{con-boundary-psi}
    \psi|_{\po \Om} =0.
\end{equation}

We define $\Sigma$ as a set for the solutions of
\eqref{psi-elliptic}:
\begin{equation}\label{def-set-sigma}
\Sigma = \{u: \|u - l\|^{(-\ga-1;P)}_{2,\ga;(\gb);\Om}  \le C^\ast
m_0 \},
\end{equation}
where  constant $C^\ast$, depending on $m_\ast, p_0, \rho_0,
\ga,\gb,\gd, \mathcal{A}$, will be determined later in the
estimates.

We state the following theorem, which implies Theorem
\ref{thm-main}:

\medskip
\begin{theorem} \label{thm-psi}
For sufficiently small $m_0$, depending on $m_\ast, p_0, \rho_0,
\ga, \gb, \gd, \mathcal{A}$, and sufficiently small $\ve$, depending
on $m_\ast, p_0, \rho_0, \ga, \gb, \gd, \mathcal{A}$ and $m_0$,
there exists a unique solution for equation \eqref{psi-elliptic}
with boundary condition \eqref{con-boundary-psi}  in the set
$\Sigma$ defined in \eqref{def-set-sigma}.
\end{theorem}

We split the proof of Theorem \ref{thm-psi} into the following
steps:

\begin{enumerate}
  \item Use bounded domain $\Om_R:=\Om\cap B_R(O)$ to approach
  $\Om$, where $B_R(O)$ is the ball with radius $R$
  and centered at the origin.
   We linearize equation \eqref{psi-elliptic} and solve the
  linear equation in $\Om_R$.
  \item With proper estimates for the linear equation,
   we solve the nonlinear equation \eqref{psi-elliptic}
  in bounded domain $\Om_R$ using Schauder fixed point theorem.
  \item Let $R \to \infty$, we prove the existence of solution for \eqref{psi-elliptic}
  in $\Om$.
  \item Estimate the difference of any two solutions in $\Om_R$ and
  then let $R \to \infty$ to obtain the uniqueness of the solution
  for \eqref{psi-elliptic}.
\end{enumerate}

Henceafter, we will use  $C$ to denote generic constants, depending
on the fixed data $m_\ast, p_0, \rho_0, \ga,\gb, \gd$, and the
profile $\mathcal{A}$, but  independent of $m_0, R$.

\section{Estimates of a linear equation}

In this section, we study a linear elliptic equation
\begin{equation}\label{eqn-lin-psi}
a_{ij}(\xx) u_{x_i x_j} + b_i(\xx) u_{x_i} + b_0 (\xx)u=0,
\end{equation}
in domain $\Om_R$. The estimates of this equation will be used later
for the linearized equation. For equation \eqref{eqn-lin-psi}, we
have the following assumptions for the coefficients:
\begin{eqnarray}
\label{con-ellip-aij} && a_{ij}(\xx)\xi_i\xi_j\ge \gl
(\xi_1^2+\xi_2^2)
 \quad \mbox{for any } \xi_i \in \R,\\
\label{con-bd-aij-bi} &&\|a_{ij}- e\gd_{ij}\|_{C^{\ga}(\Om_R)}+
 \sum_{i=1}^3
\|b_i\|_{C^\ga(\Om_R)} \le C m_0,\\
\label{con-b012}&& (x_2+1)(|b_1|+|b_2|) +(x_2 +1)^2|b_0| \le Cm_0,
\end{eqnarray}
where $\lambda, e$ are constants depending on $p_0, \rho_0$, and
$\gd_{ij}=1$ for $i=j$, otherwise $\gd_{ij}=0$.

We let $R> D_0+1$ so that the boundary of $\Om_R$ includes the whole
profile $\mathcal{A}$.
Let $S_R=\{|\xx|=R\} \cap \po \Om_R$. Now we not only have the
corner points $A_-, A_+$, but also have additional corners as the
intersection of $S_R$ with $\po \Om$. These two points are denoted
by
\[
S^R_- = \Gamma_- \cap S_R, \quad S^R_+= \Gamma_+ \cap S_R.
\]
Then we define the set of boundary points for the weight as:
\[
\tilde{P} = \{ A_-, A_+, S^R_-, S^R_+ \}.
\]

 The boundary condition for \eqref{eqn-lin-psi} is:
\begin{equation}\label{con-bd-lin-psi}
    u|_{\po \Om_R} =g,
\end{equation}
where
\begin{equation}\label{con-g}
    \|g\|^{(-1-\ga;\tilde{P})}_{2,\ga;(\ga+\gb);\Om_R} \le C m_0.
\end{equation}

We then have the following lemma
\begin{lemma} \label{lem-est-u}
Suppose $u \in C^{2,\ga}(\Om_R)\cap C(\overline{\Om_R})$ is a
solution for \eqref{eqn-lin-psi} with boundary condition
\eqref{con-bd-lin-psi} and assumptions
\eqref{con-ellip-aij}--\eqref{con-g} hold. For sufficiently small
$m_0$ independent of $R$, we have the following estimate for $u$:
\begin{equation}\label{est-u-R}
    \|u\|_{2,\ga;(\gb);\Om_R}^{(-1-\ga;\tilde{P})} \le C^\ast m_0,
\end{equation}
where $C^\ast$ is a constant independent of $m_0, R$.
\end{lemma}

To prove Lemma \ref{lem-est-u}, we need a maximum principle for the
elliptic equation \eqref{eqn-lin-psi} without restriciton on the
sign for $b_0$. We take the following lemma from \cite{hl} (Theorem
2.11):

\begin{lemma} \label{lem-compare}
Let elliptic operator $L= a_{ij}\po_{x_i}\po_{x_j} + b_i\po_{x_i} +
b_0$.  For any bounded connected domain $\mathcal{D}$, assume
$a_{ij}, b_i \in C^0(\overline{\mathcal{D}})$ and $a_{ij}$ satisfies
ellipticity condition \eqref{con-ellip-aij}. Suppose there exists a
function $v \in C^2(\mathcal{D}) \cap C^1(\overline{\mathcal{D}})$
such that $v
>0$ in $\overline{\mathcal{D}}$ and $L v \le 0$ in $\mathcal{D}$.
 Suppose $u \in
C^2(\mathcal{D}) \cap C(\overline{\mathcal{D}})$ satisfies $L u \ge
0 $ in $\mathcal{D}$. Then $\frac{u}{v}$ achieves its nonnegative
maximum on the boundary $\po \mathcal{D}$.
\end{lemma}
With this maximum principle, we start to prove Lemma
\ref{lem-est-u}:
\begin{proof}
Notice that there are two different weights in the norm in
\eqref{est-u-R} (See definition of weighted norm \eqref{def-norm}) :
the weight with upper index $-\ga-1$ is for the small scale near the
corner points $A_-, A_+, S^R_-, S^R_+$ and the weight with lower
index $\gb$ is for the large scale away from the origin. We split
the proof into three parts: {\bf Part 1} is for the estimate of
maximum norm of $u$ in the whole domain $\Om_R$;  {\bf Part 2} is
for the region near $\mathcal{A}$; and {\bf Part 3} is for the
region far away from the profile $\mathcal{A}$. Let $D=2 D_0+1$,
where $D_0$ is the radius to bound the profile $\mathcal{A}$. Let
$\Om_{D} = \Om_R \cap \{|\xx| \le D\}$ be the region for {\bf Part
2}, and $\Om^c_D=\Om_{R}\cap \{|\xx|>D-1\} $ for {\bf Part 3}.

{\bf Part 1.} In this part, we first construct a comparison function
$v$ and use the maximum principle (Lemma \ref{lem-compare}) in the
whole domain $\Om_R$ to obtain the  control of the maximum norm of
$u$.

 Define the comparison function $v$ by
\begin{equation}\label{def-compare-v}
    v(\xx) = r^{-\ga-\gb} (x_2 +1)^\ga,
\end{equation}
where $r=\sqrt{x_1^2 + (x_2 +1)^2}$.

We now verify the fact that $Lv <0$.

 First, it is easy to compute
 \begin{eqnarray}
 \nonumber  \Delta v &=& (\gb^2 - \ga^2)r^{-\ga-\gb-2}(x_2 +1)^\ga\\
 \nonumber  &&- \ga(1-\ga)r^{-\ga-\gb} (x_2 +1)^{\ga-2} \\
   &\le& - \ga(1-\ga)r^{-\ga-\gb} (x_2 +1)^{\ga-2} ,
   \label{est-v-laplace}
 \end{eqnarray}
 noticing that $0 <\gb< \ga <1$.
Also, one can verify that
\begin{eqnarray}
  \label{est-Dv} && |D v| \le C r^{-\ga-\gb}(x_2 +1)^{\ga-1},\\
  \label{est-DDv} && |D^2 v| \le C r^{-\ga-\gb}(x_2 +1)^{\ga-2}.
\end{eqnarray}
We rewrite $Lv$ as
\[
Lv= (L-e \Delta ) v + e\Delta v.
\]
By assumptions \eqref{con-bd-aij-bi} and \eqref{con-b012}, together
with \eqref{est-Dv} and \eqref{est-DDv}, we have
\begin{equation*}
   |( L-e\Delta ) v| \le C m_0 r^{-\ga-\gb}(x_2 +1)^{\ga-2}.
\end{equation*}
The above estimate and \eqref{est-v-laplace} imply that $L v <0$ in
$\Om_R$, provided $m_0$ is small enough. Obviously, $v$ is positive.
Hence, by the maximum principle (Lemma \ref{lem-compare}), we
conclude that
\[
\frac{u}{v} \le \max_{\po \Om_R}\frac{|g|}{v}.
\]
By replacing $v$ with $-v$ and using Lemma \ref{lem-compare} again,
we have
\[
|u/v| \le \max_{\po \Om_R} |g/v|.
\]
This, with assumption \eqref{con-g}, implies
\begin{equation}\label{est-u-max}
|u(\xx)| \le C m_0 r^{-\gb}.
\end{equation}

{\bf Part 2.} For the region near the profile $\mathcal{A}$, we need
to take care of the corner points $A_-, A_+$. We use the weight up
to $P$ and drop the lower index $\gb$ for the  weight away from
$\mathcal{A}$.  We treat the corner $A_-$ first, and $A_+$ can be
dealt in the same way.  For convenience, we move $A_-$ to the origin
$O$. Assume the angle between $\Gamma_-$ and $x_1$-axis at $O$
(original $A_-$) is $\theta_-$, and the angle between $\mathcal{A}$
and $x_1$-axis at $O$ is $\theta_0$. Therefore, the tangential
directions of $\Gamma_-$ and $\mathcal{A}$ at $O$ are
 \[
 \nu_- = (\cos \theta_-, \sin \theta_-), \quad \nu_0 = (\cos \theta_0, \sin
 \theta_0),
 \]
 respectively.
  Let $\bar{u} = u - g(O)-c_1 x_1 -
c_2 x_2$, where $c_1, c_2$ are linear combinations of $\frac{\po
g}{\po \nu_-}(O)$ and $\frac{\po g}{\po \nu_0}(O)$ through solving
the linear system
\begin{equation*}
   \left\{
      \begin{array}{l}
       (c_1, c_2)\cdot \nu_- =\frac{\po
g}{\po \nu_-}(O)\\
(c_1, c_2)\cdot \nu_0 =\frac{\po g}{\po \nu_0}(O).
      \end{array}
    \right.
\end{equation*}
Hence
\begin{equation*}
    |c_1| + |c_2| \le C |Dg (O)|,
\end{equation*}
and
\begin{equation}\label{con-ubar}
    \bar{u} (O) =0, \quad D \bar{u}(O) = (0,0).
\end{equation}
Choose $r_0>0$ small enough, such that $ \Om_D\cap B_{r_0}(O)$ is
connected.  Thus, for this fixed radius $r_0 < \min(D_0,1)$,
\[
\bar{u}(\xx)|_{\po \Om_D\cap B_{r_0}(O)} \le C m_0 |\xx|^{1+\ga}.
\]

We know $\bar{u}$ satisfies the following equation
\begin{equation}\label{eqn-ubar}
  \bar{L}\bar{u}\equiv  a_{ij} \po_i \po_j \bar{u} + b_i \po_i \bar{u} = F_0,
\end{equation}
where $F_0=-b_ic_i -b_0 u$. By estimate \eqref{est-u-max} and
condition \eqref{con-b012}, we have
\[
|F_0| \le Cm_0.
\]

Notice the elliptic operator in \eqref{eqn-ubar} does not contain
$b_0 \bar{u}$ term. So we can use standard maximum principle to
control $\bar{u}$. The comparison function for $\bar{u}$ is defined
in polar coordinates $(r,\theta)$ by
\[
v_1 (r,\theta)= C m_0 r^{1+\ga} \sin (\tau + (\theta - \theta_0)),
\]
for small positive $\tau$ depending on $\theta_- - \theta_0$ and $
\ga$. One can check that
\[
\bar{L} v_1 < -C m_0 < F_0 = \bar{L} \bar{u}.
\]
Also the boundary condition satisfies
\[
v_1|_{\po (\Om_D \cap B_{r_0}(O))} > C m_0 r^{1+\ga} > \bar{u}|_{\po
(\Om_D \cap B_{r_0}(O))}.
\]
By maximum principle, we conclude that
\begin{equation*}
    |\bar{u}| \le C m_0 r^{1+\ga},
\end{equation*}
for $|\xx| < r_0 $.

Once we have the above estimate near the corner $A_-$, we use
 Schauder estimates with proper scaling to obtain the
estimate near the corner $A_- (O)$:
\begin{equation} \label{est-ubar-corner}
\|u\|^{(-1-\ga;\{A_-\})}_{2,\ga;\Om_D \cap B_{\frac{r_0}{2}}(A_-)}
\le Cm_0.
\end{equation}
The procedure is standard and related details can be found in
chapter 6 of \cite{gt}. One can also refer to \cite{ccf} (Lemma 4.2)
for similar scaling argument. We sketch the proof as follows.

For any $\xx_0 \in \Om_D \cap B_{r_0/2}$, let the angle between
$\Gamma_-$ and the ray $A_-\xx_0$ be $\theta_{\xx_0}$, and the angle
between $\Gamma_-$ and $\mathcal{A}$ be $\theta^0_-$.  Consider two
cases: {\bf Case 1}, $\theta_{\xx_0} > \pi/6$ and $\theta_-
-\theta_{\xx_0}> \pi/6$ ; {\bf Case 2}, otherwise. For {\bf Case 1},
we know that $\bar{u}$ satisfies equation
\[
L \bar{u} = -b_ic_i -b_0(g(O) + c_1x_1 +c_2x_2).
\]
Take  the ball $B_{|\xx_0|/2}(\xx_0) \subset \Om_D \cap
B_{\frac{r_0}{2}}$ as the domain and by Schauder interior estimate
(Theorem 6.2, \cite{gt}), we have
\begin{equation}\label{est-ubar-int}
   \|\bar{u}\|^{(0)}_{2,\ga;B_{\frac{|\xx_0|}{2}}(\xx_0)} \le C m_0|\xx_0|^{\ga+1}.
\end{equation}
Here the upper index $(0)$ is understood as the weight up to $\po
B_{\frac{|\xx_0|}{2}}(\xx_0)$.

For {\bf Case 2}, let $\xx^\ast$ be a boundary point with the
shortest distance from $\xx_0$, and $d^\ast =
|\xx_0|\sin(\frac{3}{4}\theta_-^0)$.  Hence,
$B_{\frac{7d^\ast}{8}}(\xx^\ast)$ still contains $\xx^\ast$. We use
Schauder boundary estimate (Lemma 6.4, \cite{gt}) in the domain
$\Om_D \cap B_{d^\ast}(\xx^\ast)$ to obtain the estimate:
\begin{equation}\label{est-ubar-bd}
 \|\bar{u}\|^{(0)}_{2,\ga;B_{d^\ast}(\xx^\ast)} \le C m_0|\xx_0|^{\ga+1}.
\end{equation}
Combining \eqref{est-ubar-int} and \eqref{est-ubar-bd} gives the
corner estimate \eqref{est-ubar-corner}.  The other corner $A_+$ is
treated in the same way. Together with standard Schauder estimates
away from the corners, we conclude the estimate in $\Om_D$:
\begin{equation} \label{est-u-D}
\| u\|_{2,\ga; \Om_D}^{(-\ga-1;P)} \le C m_0.
\end{equation}


{\bf Part 3.}
 For the domain $\Om_D^c$, we also consider two kinds of
estimates: one is near the corner points $S^R_-, S^R_+$, the other
is away from the corners.

The corner estimates are similar to those in {\bf Part 2}. In brief,
consider the corner $S^R_-$ for instance. If $\xx_0 \in \Om^c_D \cap
B_{r_0}(S^R_-)$, we have the following the estimate
\begin{equation} \label{est-u-cornerS}
\|u\|^{(-1-\ga;\{S^R_-\})}_{2,\ga;\Om_D^c \cap
B_{\frac{r_0}{2}}(S_-^R)} \le C m_0 R^{-\gb}.
\end{equation}

If $\xx_0 \in \Om_D^c$  and $ |\xx_0| < R/2$ ,   the ball $
B_{\frac{|\xx_0|}{2}} (\xx_0)$ has no intersection with the outer
boundary $S_R = \{|\xx| =R\}$ or the profile $\mathcal{A}$. Using
conditions \eqref{con-ellip-aij},\eqref{con-bd-aij-bi},
\eqref{con-g} and estimate \eqref{est-u-max}, by Schauder interior
 estimates (see Theorem 6.2 in \cite{gt}),
we have
\begin{equation} \label{est-u-Dcint}
\|u\|_{2,\ga;(\gb);B_{\frac{|\xx_0|}{4}}(\xx_0)} \le C m_0,
\end{equation}
where no upper index in the norm means no weight up to $\tilde{P}$.
For $|\xx_0|\ge R/2$, we use Schauder boundary estimates (Lemma 6.4
in \cite{gt}) with the boundary condition \eqref{con-g} and estimate
\eqref{est-u-cornerS}  to obtain
\begin{equation} \label{est-u-Dcbd}
\|u\|^{(-1-\ga; \{S^R_-,
S^R_+\})}_{2,\ga;(\gb);B_{\frac{|\xx_0|}{4}}(\xx_0)\cap \Om_D^c} \le
C m_0,
\end{equation}
Estimates \eqref{est-u-D}, \eqref{est-u-Dcint} and
\eqref{est-u-Dcbd} imply estimate \eqref{est-u-R} in the lemma.

\end{proof}

By the continuity method, one can prove the existence of solutions
for \eqref{eqn-lin-psi} with estimate \eqref{est-u-R}.
The uniqueness is simply the result of the maximum principle, Lemma
\ref{lem-compare}, with the aid of the comparison function $v$
constructed in Lemma \ref{lem-est-u}. Since the procedure is
standard, we omit the proof and only state the result as follows.

\begin{lemma}
Assume \eqref{con-ellip-aij}--\eqref{con-g} holds. For sufficiently
small $m_0$, equation \eqref{eqn-lin-psi} with boundary condition
\eqref{con-bd-lin-psi} admits a unique solution $u \in C^2(\Omega_R)
\cap C(\overline{\Om_R})$.
\end{lemma}

\section{Nonlinear Equation in Bounded Domain}
In this section, we will solve equation \eqref{psi-elliptic} in the
bounded domain $\Om_R$ with boundary condition given below.

We want to prescribe the boundary data for $\psi$ such that
\eqref{con-boundary-psi} hold on $\po \Om \cap \Om_R$ and $\psi -l$
vanishes away from $\po \Om$. We will define function $g$ such that
\begin{equation} \label{con-bd-psi-l}
\psi -l = g  \quad \mbox{ on } \quad \po \Om_R.
\end{equation}
 First it is easy to construct a smooth
cutoff function $\eta (s) $ such that $\eta(s) =1$ for $|s|\le D_0$
and $\eta(s) =0$ for $|s|\ge D_0 +1 $. We also assume that
$\|\eta\|_{C^{2,\ga}(\R)} \le 10$. Let
\begin{equation}\label{def-bd-g}
    g(\xx) = -\eta(x_2) \left( (1-\eta(x_1))\,l(f_{\mbox{sign}(x_1)}(x_1))
+ \eta(x_1) \,l(x_2) \right),
\end{equation}
where $\mbox{sign}(x_1) = -$ for $x_1 <0$ and $\mbox{sign}(x_1) = +$
for $x_1 >0$.

It is easy to check that $g|_{\po \Om} = -l|_{\po \Om}$, $g =0$ for
$x_2 > D_0+1$, and also $g$ satisfies condition \eqref{con-g}. Now,
let constant $C^\ast$ in \eqref{def-set-sigma} be the same as
 in estimate \eqref{est-u-R} in Lemma \ref{lem-est-u}. The set for
solutions of \eqref{psi-elliptic} in $\Om_R$ is
\begin{equation}\label{def-set-sigmaR}
    \Sigma_R = \{u: \|u - l\|^{(-\ga-1;\tilde{P})}_{2,\ga;(\gb);\Om_R}
  \le C^\ast
m_0 \}.
\end{equation}

 We state our lemma for the solution of
\eqref{psi-elliptic} in $\Om_R$:

\begin{lemma} \label{lem-psi-OmR}
Equation \eqref{psi-elliptic} in $\Om_R$ with boundary condition
\eqref{con-bd-psi-l} admits a unique solution in the set $\Sigma_R$,
provided $m_0, \ve$  are sufficiently small.
\end{lemma}

\begin{proof}

To prove the lemma, we first linearize the nonlinear equation
\eqref{psi-elliptic}. By solving the linearized equation, we
construct a map $T$ in the set $\Sigma_R$ .
 The solution of the nonlinear equation
\eqref{psi-elliptic} is a fixed point of $T$.

We start the proof with the linearization of \eqref{psi-elliptic}.
We know that the limit function $l$ satisfies \eqref{psi-elliptic}
by its definition. That means the following equation holds:
\begin{equation}\label{eqn-l}
    a_{ij}(l, 0, l')\,l_{x_i x_j} = F (l, 0, l'),
\end{equation}
where $a_{ij}, F$ are defined in \eqref{def-a11}--\eqref{def-a22}.
Taking the difference of equations \eqref{psi-elliptic} and
\eqref{eqn-l} leads to
\begin{eqnarray}\nonumber
   && a_{ij}(\psi, \grad \psi)(\psi - l)_{x_i x_j} + l''(a_{22}(\psi, \grad
\psi)-a_{22}(l,0, l'))\\
&=&F(\psi, \grad \psi) -F(l, 0, l').\label{eqn-psi-l}
\end{eqnarray}
Denote $(l+s(\psi -l), \grad ( l+s(\psi -l)))$ by
$\mathbf{t}^\psi_s$, and let
\begin{eqnarray}
  a_{ij}^\psi &=& a_{ij}(\psi, \grad \psi) ,\\
  b_0^\psi&=& l''\int_0^1 (a_{22})_\psi(\mathbf{t}^\psi_s)d\,s -
\int_0^1 F_\psi (\mathbf{t}^\psi_s)d\,s,\\
 (b_1^\psi, b_2^\psi)&=&l''\int_0^1 (a_{22})_{\grad \psi}(\mathbf{t}^\psi_s)d\,s -
\int_0^1 F_{\grad \psi} (\mathbf{t}^\psi_s)d\,s. \label{def-bipsi}
\end{eqnarray}
Then we linearize equation \eqref{eqn-psi-l} as follows:
\begin{equation}\label{eqn-psitil-l-lin}
    a^\psi_{ij} (\tilde{\psi} - l )_{x_i x_j}
 + b_i^\psi(\tilde{\psi} - l) + b_0^\psi (\tilde{\psi} - l)=0.
\end{equation}
We solve the above equation by applying Lemma \ref{lem-est-u}. In
the following, we will check the conditions
\eqref{con-ellip-aij}--\eqref{con-b012} for $\psi \in \Sigma_R$. In
fact, \eqref{con-ellip-aij} will be satisfied if
\eqref{con-bd-aij-bi} holds with sufficiently small $m_0$. We let
$e$ in \eqref{con-bd-aij-bi} be $a_{11} (m_0x_2, 0, m_0)=a_{22}
(m_0x_2, 0, m_0) = \gc p_0 \rho_0$. By the expressions for $a_{ij}$,
\eqref{def-a11}--\eqref{def-a22}, and \eqref{def-rhoDpsi},
\eqref{def-rhoDXi},  it is not hard to verify that
\begin{eqnarray*}
   && \|(a_{ij})_\psi(\mathbf{t}_s^\psi)\|_{C^\ga (\Om_R)}
\le \frac{\ve} {m_0} C C^\ast,\\
 && \|(a_{ij})_{\grad \psi}(\mathbf{t}_s^\psi)\|_{C^\ga (\Om_R)}
 \le \frac{\ve} {m_0} C C^\ast,
\end{eqnarray*}
for $\psi \in \Sigma$ and $m_0$ small.

In the proof of this lemma, the generic constants $C$ are
independent of $C^\ast$.

 Let $\ve < \frac{m_0}{(C^\ast)^2}$,
we have
\begin{eqnarray*}
 && \|a_{ij}^\psi - e \gd_{ij}\|_{C^\ga \Om_R}\\
&\le & \int_0^1 \|(a_{ij})_{\psi}(\mathbf{t}_s^\psi)\|_{C^\ga
(\Om_R)} ds \|\psi -m_0x_2\|_{C^\ga (\Om_R)} \\
&&+ \int_0^1 \|(a_{ij})_{\grad \psi}(\mathbf{t}_s^\psi)\|_{C^\ga
(\Om_R)} ds \|\grad(\psi -m_0x_2)\|_{C^\ga (\Om_R)}\\
&\le & \ve C (C^\ast)^2  \\
&\le& Cm_0.
\end{eqnarray*}
In the same manner, we can obtain \[ \|b_{i}\|_{C^\ga (\Om_R)} \le
Cm_0, \quad i=0,1,2.
\]

 The above estimates  lead to condition \eqref{con-bd-aij-bi}.
Now we verify condition \eqref{con-b012}. For $b^\psi_0$, by its
expression, we need to estimate $(a_{22})_\psi$ and $F_\psi$. By the
definition of $F$, \eqref{def-F}, and estimates \eqref{est-A},
\eqref{est-B}, we have
\[
|F_\psi (\psi, \grad \psi) | \le  \frac{\ve m_0 C}{(m_0 +|\psi|)^2}.
\]
Let $u_s = l + s(\psi -l)$. For any $\psi \in \Sigma$, we consider
two cases: {\bf Case 1},  $|\psi -l| \le \frac{1}{4} m_0 x_2$; {\bf
Case 2}, otherwise. For {\bf Case 1},
\[ u_s \ge l - \frac{1}{4} m_0 x_2 \ge \frac{1}{4} m_0 x_2,
 \]
noticing \eqref{est-l-asy}. Therefore,
\begin{eqnarray*}
|F_\psi( \mathbf{t}^\psi_s)| &=& |F_\psi(u_s, \grad u_s)| \le
\frac{\ve m_0 C}{(m_0 +|u_s|)^2} \le \frac{\ve C m_0}{(1+x_2)^2},
\end{eqnarray*}
for $\ve < {m_0}^2$. For {\bf Case 2}, i.e.,  $|\psi - l| >
\frac{1}{4}m_0 x_2$, since $\psi \in \Sigma_R$, we have
\[
 \frac{1}{4}m_0 x_2 < |\psi -l| \le \frac{C^\ast m_0}{|\xx|^\gb}.
\]
This implies that $x_2 < (4C^\ast)^{\frac{1}{1+\gb}} \equiv R_0$.
Hence, for $\ve < ((m_0/(1+R_0))^2$, we have
\[
|F_\psi( \mathbf{t}^\psi_s)| \le \frac{\ve C}{m_0} \le \frac{\ve C
m_0 (1+ R_0)^2}{(m_0)^2 (1+x_2)^2} <\frac{ C m_0}{(1+x_2)^2}.
\]
 The above analysis about both $\bf Case 1$ and {\bf Case 2} gives rise to
\begin{equation}\label{est-int-FDpsi}
\left|\int_0^1 F_\psi (\mathbf{t}^\psi_s)d\,s\right| \le \frac{ C
m_0}{(1+x_2)^2}.
\end{equation}

Similarly, we have
\[
|(a_{22})_\psi( \mathbf{t}^\psi_s)| \le \frac{ C m_0}{1+x_2}.
\]
Together with the fact that $|l''(x_2)| < \frac{\ve C m_0}{1+x_2}$,
we conclude that
\[
|b_0^\psi| < \frac{C m_0}{(1+x_2)^2}.
\]
Same arguments apply to the estimates for $b^\psi_1, b^\psi_2$:
\[
|b_1^\psi| + |b_2^\psi| < \frac{ C m_0}{1+x_2}.
\]
Therefore, we have verified condition \eqref{con-b012}. Thus, we
apply Lemma \ref{lem-est-u} to solve \eqref{eqn-psitil-l-lin} for
$\tilde{\psi}$, where $u=\tilde{\psi} -l$ in Lemma \ref{lem-est-u}.
By estimate \eqref{est-u-R}, we know that $\tilde{\psi} \in
\Sigma_R$. Therefore, we can define a map $T$ from $\Sigma_R$ to
itself by $T\psi \equiv \tilde{\psi}$. It is obvious that a solution
 for the nonlinear equation \eqref{psi-elliptic} is a fixed point
of $T$. In order to prove the existence of a fixed point of $T$, we
apply Schauder fixed point theorem, which says: if $\Sigma_R$ is a
compact convex set of a Banach space $\mathcal{B}$, and $T: \Sigma_R
\to \Sigma_R$ is continuous in $\mathcal{B}$, $T$ has a fixed point.

Now we let $\mathcal{B} = C^{(-1-\ga';\tilde{P})}_{2,\ga';(\gb)}(
\Om_R)$ (cf. definition \eqref{def-spaceHolder}), where $0<\ga'<
\ga$. Obviously, $\Sigma_R$ is compact and convex in $\mathcal{B}$.
We only need to verify $T$ is continuous in $\Sigma_R$. We prove
this by contradiction argument.

Suppose $T$ is not continuous. Then there exists a sequence
$\{\psi_n\} \subset \Sigma_R$ such that $\psi_n \to \psi$ in
$\mathcal{B}$, but $T\psi_n$ does not converge to $\tilde{\psi} = T
\psi$. This implies that we can find a subsequence $\{
T\psi_{n_k}\}$ such that
\begin{equation} \label{diff-psi}
\| T \psi_{n_k} - \tilde{\psi}\|_{\mathcal{B}} \ge c_0 >0,
\end{equation}
 where $\|
\cdot \|_{\mathcal{B}}$ denotes the weighed norm $\|
\cdot\|^{(-1-\ga';\tilde{P})}_{2,\ga';(\gb);\Om_R}$ and $c_0$ is a
fixed constant.  Since $\{T \psi_{n_k}\}$ is compact in
$\mathcal{B}$, there exists a subsequence, still denoted by  $\{T
\psi_{n_{k}}\}$, convergent to $\bar{\psi} \in \Sigma_R$. On the
other hand, $\psi_n \to \psi$ in $\mathcal{B}$ implies that
$a^{\psi_{n_{k}}}_{ij} \to a^\psi_{ij} (i,j=1,2) $,
$b_i^{\psi_{n_{k}}} \to b_i^\psi (i=0,1,2)$ in $C^{\ga'}$ norm. Let
$k \to \infty$  and we see that $\bar{\psi}$ is also a solution of
\eqref{eqn-psitil-l-lin} with the same boundary condition
\eqref{con-bd-lin-psi}. Inequality \eqref{diff-psi} implies that
$\|\tilde{\psi} -\bar{\psi}\|_{\mathcal{B}} \ge c_0 >0$, which means
$\tilde{\psi}, \bar{\psi}$ are two distinct solutions for
\eqref{eqn-psitil-l-lin}. This contradicts with the uniqueness of
the solution for \eqref{eqn-psitil-l-lin}. Hence, we verified the
continuity of $T$ in $\Sigma_R$.

By Schauder fixed point theorem, there exists a fixed point $\psi_R$
for $T$.  So $\psi_R$ is a solution for the nonlinear equation
\eqref{psi-elliptic} in $\Om_R$ with boundary condition
\eqref{con-bd-psi-l}.

The uniqueness of the solution for \eqref{psi-elliptic} is proved by
the maximum principle Lemma \ref{lem-compare} as follows.

For any two solutions $\psi_1, \psi_2 \in \Sigma_R$ of
\eqref{psi-elliptic}, \eqref{con-bd-psi-l}, we take the difference
of the two equations and obtain
\[
a_{ij}^{\psi_1} (\psi_1 - \psi_2)_{x_i x_j} +
(\psi_2)_{x_ix_j}(a_{ij}^{\psi_1}-a_{ij}^{\psi_2} )=F(\psi_1, \grad
\psi_1) - F(\psi_2, \grad \psi_2).
\]
Similar to the notations in \eqref{eqn-psi-l}--\eqref{def-bipsi}, we
set
\[
\mathbf{t}_s=(\psi_2 + s(\psi_1 -\psi_2),\grad (\psi_2 + s(\psi_1
-\psi_2) )),
\]
and let $u=\psi_1 - \psi_2$. Then we derive the following equation:
\begin{equation}\label{eqn-psi12}
    a_{ij}^{\psi_1} u_{x_i x_j} + b_i u_{x_i} + b_0 u=0,
\end{equation}
where
\begin{eqnarray}\label{def-b0psi12}
 b_0&=&(\psi_2)_{x_ix_j}\int_0^1 (a_{ij})_\psi(\mathbf{t}_s)d s -
\int_0^1 F_\psi (\mathbf{t}_s)d s,\\
 (b_1, b_2)&=&(\psi_2)_{x_ix_j}\label{def-b12psi12}
\int_0^1 (a_{ij})_{\grad \psi}(\mathbf{t}_s)d s - \int_0^1 F_{\grad
\psi} (\mathbf{t}_s)d s. \label{def-bipsi12}
\end{eqnarray}
Notice that the factor $(\psi_2)_{x_ix_j}$ in \eqref{def-b0psi12}
and \eqref{def-b12psi12} blows up at the corner points in
$\tilde{P}= \{A_-, A_+, S^R_-, S^R_+\}$.  Lemma \ref{lem-compare}
requires continuity of  coefficients $a_{ij}, b_i$ up to the
boundary. Therefore, in order to apply the lemma,  we truncate small
neighborhoods around corners from $\Om_R$. Define
\begin{equation}\label{def-BPr}
    B_{\mathcal{P},r} = \bigcup_{I \in \mathcal{P}} B_r(I),
\end{equation}
where $\mathcal{P}$ is a set of boundary points on $\Om_R$. Let
\begin{equation}\label{def-OmRr}
    \Om_R^{\mathcal{P}, r} = \Om_R \backslash \overline{B_{\mathcal{P},r}
}, \quad S_{\mathcal{P},r} = \po B_{\mathcal{P},r} \cap \Om_R.
\end{equation}
Now, we know $b_i \in C\left(\overline{\Om_R^{\tilde{P}, r}}\right),
i=0,1,2 $.  Define $r_I = |\xx - I|$, for any corner point $I \in
\tilde{P}$. Then we have the following estimates for $b_i$:
\begin{equation}\label{est-bipsi12}
(x_2+1)(|b_1|+|b_2|) +(x_2 +1)^2|b_0| \le
Cm_0(1+\sum_{I\in\tilde{P}}r_I^{\ga-1}).
\end{equation}
We construct a comparison function $\tilde{v}$ as follows: Let
\begin{equation}\label{def-vI}
v_I = r_I^{-\frac{3}{4} \gb}(x_2 +1)^{\frac{\gb}{2}}.
\end{equation}
 Define
$\tilde{v} = \sum_{I\in\tilde{P}} v_I$.  By
\eqref{est-v-laplace}--\eqref{est-DDv}, we can verify that
\begin{eqnarray}
 a_{ij}^{\psi_1} (v_I)_{x_i x_j} \le -c_0 (r_I^{-\frac{3}{4}\gb
-2}(x_2 +1)^{\frac{\gb}{2}} + r_I^{-\frac{3}{4}\gb }(x_2
+1)^{\frac{\gb}{2}-2} ) \label{est-aijv},&&\\
\left|\sum_{i=1,2}b_i(v_I)_{x_i}+ b_0v_I \right| \le Cm_0 (
1+\sum_{I\in\tilde{P}}r_I^{\ga-1})r_I^{-\frac{3}{4}\gb}(x_2
+1)^{\frac{\gb}{2}-2},&& \label{est-biDv}
\end{eqnarray}
where $c_0 >0$ is a constant only depending on $\gb$. Inequalities
\eqref{est-aijv} and \eqref{est-biDv} directly imply that
\[
a_{ij}^{\psi_1} \tilde{v}_{x_ix_j} +b_i \tilde{v}_{x_i}+ b_0
\tilde{v} <0.
\]
 By Lemma \ref{lem-compare}, we have
\begin{equation}\label{est-u/vtil}
\sup_{\Om_R^{\tilde{P}, r}} \frac{|u|}{\tilde{v}} = \max_{\po
\Om_R^{\tilde{P}, r}} \frac{|u|}{\tilde{v}}.
\end{equation}

Since $u|_{\po \Om_R}= 0$ and $u\in C^{1,\ga}(\Om_R) $, we know that
\[
\max_{\po \Om_R^{\tilde{P}, r}} \frac{|u|}{\tilde{v}} = \max_{
S_{\tilde{P}, r}} \frac{|u|}{\tilde{v}} \le Cr^\ga
R^{\frac{3}{4}\gb}.
\]
Hence, we have
\begin{equation*}
    \sup_{\Om_R^{\tilde{P}, r}} |u|
 \le  C R^{\frac{3}{4}\gb}r^\ga \sup_{\Om_R^{\tilde{P}, r}}\tilde{v}
\le C R^{\frac{3}{4}\gb}r^\ga (1+ r^{-\frac{3}{4}\gb}).
\end{equation*}
Therefore, by letting $r\to 0$,  the above inequality implies
$\sup_{\Om_R} |u| \le 0$. This shows the uniqueness of the solution
for \eqref{psi-elliptic} in $\Om_R$. Hence, the proof of this lemma
is complete.

\end{proof}

\section{Subsonic Flow in Half Plane}

After we solve \eqref{psi-elliptic} in $\Om_R$, we let $R$ tend to
infinity to prove Theorem \ref{thm-psi} as follows.
\begin{proof}[Proof of Theorem \ref{thm-psi}]
By Lemma \ref{lem-psi-OmR}, for a given radius $R$, we can find a
unique solution $\psi_R \in \Sigma_R$.  By a diagonal process, one
can choose proper sequence $R_n \to \infty$ as $n \to \infty$, such
that $\psi_{R_n}$ converges to some function $\psi$ in $ \| \cdot
\|^{(-\ga'-1;P)}_{2,\ga';(\gb);\Om_Q}$ norm, for any fixed $Q >
D_0+1$. Since $\psi_{R_n} \in \Sigma_{R_n}$, we have
\begin{equation}\label{est-psi-r/2}
 \|\psi_{R_n} - l\|^{(-\ga-1;P)}_{2,\ga;(\gb);\Om_{\frac{R_n}{2}}}  \le C^\ast m_0,
\end{equation}
for any $R_n > 2(D_0+1)$. Let $n \to \infty$ in \eqref{est-psi-r/2},
we obtain estimate
\begin{equation}\label{est-psi-Om}
 \|\psi - l\|^{(-\ga-1;P)}_{2,\ga;(\gb);\Om}  \le C^\ast m_0,
\end{equation}
which implies  $\psi \in \Sigma$. This completes the existence of
solutions for \eqref{psi-elliptic}.

To prove the uniqueness of the solution, we will use the asymptotic
behavior of solutions described by set $\Sigma$. We still use the
truncated domain $\Om_R$ and follow the same strategy as in the
uniqueness part of Lemma \ref{lem-psi-OmR}. Now, the situation here
is slightly different from that in Lemma \ref{lem-psi-OmR}: (1) we
have no singularity for $(\psi_2)_{x_i x_j}$ at corners $S^R_-,
S^R_+$; (2) $u= \psi_1 -\psi_2$ does not vanish on the boundary
portion $S_R= \{|\xx|=R\}\cap \Om$.

Similarly as in Lemma \ref{lem-psi-OmR}, we have the same equation
\eqref{eqn-psi12} for $u$. We define the comparison function
$\bar{v}$ by $\bar{v} = v_{A_-} + v_{A_+}$, where $v_{A_-}, v_{A_+}$
are defined in \eqref{def-vI}. The domain we consider here is
$\Om_R^{P, r}$ defined by \eqref{def-BPr}.  By the same computation
as in Lemma \ref{lem-psi-OmR}, we have
\[
a_{ij}^{\psi_1} \bar{v}_{x_ix_j} +b_i \bar{v}_{x_i}+ b_0 \bar{v} <0.
\]
We know that $|u| \le CR^{-\gb}$ on $S_R$ by the definition of
$\Sigma$. Also, we see
\[
|u|_{S_{P,r}}| \le Cr^\ga, \quad  \bar{v}|_{S_R} \ge
R^{-\frac{3}{4}\gb}, \quad  \bar{v}|_{S_{P,r}} \ge
r^{-\frac{3}{4}\gb}.
\]
Therefore, we conclude that
\begin{equation}\label{est-u/vbar}
\sup_{ \Om_R^{P, r}} \frac{|u|}{\tilde{v}}  \le C \max (r^{\ga
+\frac{3}{4}\gb}, R^{-\frac{1}{4}\gb}).
\end{equation}
Let $r=R^{-\frac{1}{6}}$ and we obtain
\[
\sup_{\Om_R^{P, r}} |u|
 \le CR^{-\frac{1}{4}\gb}\sup_{\Om_R^{P, r}} \bar{v}
\le CR^{-\frac{1}{4}\gb} (1+ R^{\frac{\gb}{8}}) \le
CR^{-\frac{\gb}{8}}.
\]
Letting $R \to \infty$ gives rise to $\sup_{\Om}|u| \le 0$, which
implies the uniqueness of the solution for \eqref{psi-elliptic} in
$\Sigma$. This finishes the proof of Theorem \ref{thm-psi}.
\end{proof}

Once we proved Theorem \ref{thm-psi}, define $U=(\mm, p,\rho)$ by
\[
\mm=(\psi_{x_2}, -\psi_{x_1}), \quad \rho=\rho(\psi, \grad \psi),
\quad p= \frac{\gc-1}{\gc}A(\psi)\rho^\gc,
\]
where $\rho$ is uniquely solved from Bernoulli's law
\eqref{bernoulli-law}. Hence, by \eqref{eqn-M1}, \eqref{eqn-M2} and
\eqref{bernoulli-law}, we can recover the original Euler equations
\eqref{eqn-Euler1}. It is easy to check that $U$ satisfies
\eqref{est-U-Uinfty} with the aid of estimate \eqref{est-psi-Om}.

\begin{remark} \label{rem-unique}
The uniqueness of the solution for Euler system \eqref{eqn-Euler1}
can not be obtained from Theorem \ref{thm-psi} due to two obstacles.
One is the existence of stagnation points, which disqualifies
equivalence between  equation \eqref{psi-elliptic} and  the two
momentum equations \eqref{eqn-M1}, \eqref{eqn-M2}. The corners $A_-,
A_+$ on $\po \Om$ are necessarily stagnation points, because $\grad
\psi$ is continuous up to the corners. Whether or what kind of
stagnation points may appear inside domain $\Om$ is not clear.  The
other problem is about the complexity of streamline topology. During
the reduction in section \ref{sec-reduce-system}, we assume
streamlines have simple topology, which means that streamlines in
$\Om$ extend from $-\infty$ to $\infty$ in $x_1$, so that
information about $A,B$ can be carried along streamlines and reach
the whole domain $\Om$. However, the geometry of profile
$\mathcal{A}$ may be complicated and cause nontrivial topology of
streamlines, such as closed orbits or intersection of streamlines at
stagnation points. The above reasons prevent us from obtaining the
uniqueness for the Euler flows out of Theorem \ref{thm-psi}.
\end{remark}

\noindent {\bf Acknowledgments}. \quad I would like to thank Gui-Qiang Chen,
 Mikhail Feldman and Marshall Slemrod for their constructive suggestions
 and stimulating discussions.
 The paper
was supported in part by the
National Science Foundation under Grants DMS-0354729.

\end{document}